\documentclass{article}
\usepackage{amsfonts}
\usepackage{amsmath}
\usepackage{amssymb}
\usepackage{graphics}
\usepackage[pdftex]{graphicx}
\usepackage[round,comma,authoryear]{natbib}

\newtheorem{theorem}{Theorem}

\newtheorem{proposition}[theorem]{Proposition}
\newtheorem{remark}[theorem]{Remark}

\newenvironment{proof}[1][Proof]{\noindent\textbf{#1.} }{\ \rule{0.5em}{0.5em}}
\input{tcilatex}

\defcitealias{wanglijia07}{Wang, G. Li and Jiang (2007)}
\defcitealias{wanglitsa07}{Wang, G. Li and Tsai (2007)}
\defcitealias{wanrlitsa07}{Wang, R. Li and Tsai (2007)}

\begin{document}

\title{On the Distribution of the Adaptive LASSO Estimator}
\author{Benedikt M. P\"otscher and Ulrike Schneider \\
Department of Statistics, University of Vienna}
\date{First version: December 2007\\
This version: December 2008}
\maketitle

\begin{abstract}
We study the distribution of the adaptive LASSO estimator (\cite{zou06}) in
finite samples as well as in the large-sample limit. The large-sample
distributions are derived both for the case where the adaptive LASSO
estimator is tuned to perform conservative model selection as well as for
the case where the tuning results in consistent model selection. We show
that the finite-sample as well as the large-sample distributions are
typically highly non-normal, regardless of the choice of the tuning
parameter. The uniform convergence rate is also obtained, and is shown to be
slower than $n^{-1/2}$ in case the estimator is tuned to perform consistent
model selection. In particular, these results question the statistical
relevance of the `oracle' property of the adaptive LASSO estimator
established in \cite{zou06}. Moreover, we also provide an impossibility
result regarding the estimation of the distribution function of the adaptive
LASSO estimator. The theoretical results, which are obtained for a
regression model with orthogonal design, are complemented by a Monte Carlo
study using non-orthogonal regressors.

\textit{MSC 2000 subject classification}. Primary 62F11, 62F12, 62E15,
62J05, 62J07.

\textit{Key words and phrases}. Penalized maximum likelihood, LASSO,
adaptive LASSO, nonnegative garotte, finite-sample distribution, asymptotic
distribution, oracle property, estimation of distribution, uniform
consistency.
\end{abstract}


\section{Introduction}


Penalized maximum likelihood estimators have been studied intensively in the
last few years. A prominent example is the least absolute selection and
shrinkage (LASSO) estimator of \cite{tib96}. Related variants of the LASSO
include the Bridge estimators studied by \cite{frafri93}, least angle
regression (LARS) of \cite{efretal04}, or the smoothly clipped absolute
deviation (SCAD) estimator of \cite{fanli01}. Other estimators that fit into
this framework are hard- and soft-thresholding estimators. While many
properties of penalized maximum likelihood estimators are now well
understood, the understanding of their distributional properties, such as
finite-sample and large-sample limit distributions, is still incomplete. The
probably most important contribution in this respect is \cite{knifu00} who
study the asymptotic distribution of the LASSO estimator (and of Bridge
estimators more generally) when the tuning parameter governing the influence
of the penalty term is chosen in such a way that the LASSO acts as a
conservative model selection procedure (that is, a procedure that does not
select underparameterized models asymptotically, but selects
overparameterized models with positive probability asymptotically). In \cite%
{knifu00}, the asymptotic distribution is obtained in a fixed-parameter as
well as in a standard local alternatives setup. This is complemented by a
result in \cite{zou06} who considers the fixed-parameter asymptotic
distribution of the LASSO when tuned to act as a consistent model selection
procedure. Another contribution is \cite{fanli01} who derive the
fixed-parameter asymptotic distribution of the SCAD estimator when the
tuning parameter is chosen in such a way that the SCAD estimator performs
consistent model selection; in particular, they establish the so-called
`oracle' property for this estimator. \cite{zou06} introduced a variant of
the LASSO, the so-called adaptive LASSO estimator, and established the
`oracle' property for this estimator when suitably tuned. Since it is
well-known that fixed-parameter (i.e., pointwise) asymptotic results can
give a wrong picture of the estimator's actual behavior, especially when the
estimator performs model selection (see, e.g., \cite{kab95}, or \cite%
{leepoe05, poelee07, leepoe08a}), it is important to take a closer look at
the actual distributional properties of the adaptive LASSO estimator.

In the present paper we study the finite-sample as well as the large-sample
distribution of the adaptive LASSO estimator in a linear regression model.
In particular, we study both the case where the estimator is tuned to
perform conservative model selection as well as the case where it is tuned
to perform consistent model selection. We find that the finite-sample
distributions are highly non-normal (e.g., are often multimodal) and that a
standard fixed-parameter asymptotic analysis gives a highly misleading
impression of the finite-sample behavior. In particular, the `oracle'
property, which is based on a fixed-parameter asymptotic analysis, is shown
not to provide a reliable assessment of the estimators' actual performance.
For these reasons, we also obtain the large-sample distributions of the
above mentioned estimators under a general \textquotedblleft moving
parameter\textquotedblright\ asymptotic framework, which much better
captures the actual behavior of the estimator. [Interestingly, it turns out
that in case the estimator is tuned to perform consistent model selection a
\textquotedblleft moving parameter\textquotedblright\ asymptotic framework
more general than the usual $n^{-1/2}$-local asymptotic framework is
necessary to exhibit the full range of possible limiting distributions.]
Furthermore, we obtain the uniform convergence rate of the adaptive LASSO
estimator and show that it is slower than $n^{-1/2}$ in the case where the
estimator is tuned to perform consistent model selection. This again exposes
the misleading character of the `oracle' property. We also show that the
finite-sample distribution of the adaptive LASSO estimator cannot be
estimated in any reasonable sense, complementing results of this sort in the
literature such as \cite{leepoe06a, leepoe06b, poelee07, leepoe08b} and \cite%
{poe06}.

Apart from the papers already mentioned, there has been a recent surge of
publications establishing the `oracle' property for a variety of penalized
maximum likelihood or related estimators (e.g., \cite{bun04, bunmck05,
fanli02, fanli04, lilia07, wanlen07}, %
\citetalias{wanglijia07,wanglitsa07,wanrlitsa07}, \cite{yualin07, zhalu07,
zouyua08, zouli08, johetal08}). The `oracle' property also paints a
misleading picture of the behavior of the estimators considered in these
papers; see the discussion in \cite{leepoe05, yan05, poe07, poelee07,
leepoe08a}.

The paper is organized as follows. The model and the adaptive LASSO
estimator are introduced in Section~\ref{mo_est}. In Section \ref{theory} we
study the estimator theoretically in an orthogonal linear regression model.
In particular, the model selection probabilities implied by the adaptive
LASSO estimator are discussed in Section~\ref{prob}. Consistency, uniform
consistency, and uniform convergence rates of the estimator are the subject
of Section~\ref{cons&ucons}. The finite-sample distributions are derived in
Section~\ref{finite}, whereas the asymptotic distributions are studied in
Section~\ref{asydistribs}. We provide an impossibility result regarding the
estimation of the adaptive LASSO's distribution function in Section~\ref{imp}%
. Section \ref{simulation} studies the behavior of the adaptive LASSO
estimator by Monte Carlo without imposing the simplifying assumption of
orthogonal regressors. We finally summarize our findings in Section~\ref%
{conl}. Proofs and some technical details are deferred to an appendix.


\section{The Adaptive LASSO Estimator \label{mo_est}}


We consider the linear regression model 
\begin{equation}
Y=X\theta +u  \label{model}
\end{equation}%
where $X$ is a nonstochastic $n\times k$ matrix of rank $k$ and $u$ is
multivariate normal with mean zero and variance-covariance matrix $\sigma
^{2}I_{n}$. Let $\hat{\theta}_{LS}=(X^{\prime }X)^{-1}X^{\prime }Y$ denote
the least squares (maximum likelihood) estimator. The adaptive LASSO
estimator $\hat{\theta}_{A}$ is defined as the solution to the minimization
problem%
\begin{equation}
(Y-X\theta )^{\prime }(Y-X\theta )+2n\mu _{n}^{2}\sum_{i=1}^{k}|\theta
_{i}|/|\hat{\theta}_{LS,i}|  \label{alasso}
\end{equation}%
where the tuning parameter $\mu _{n}$ is a positive real number. As long as $%
\hat{\theta}_{LS,i}\neq 0$ for every $i$, the function given by (\ref{alasso}%
) is well-defined and strictly convex and hence has a uniquely defined
minimizer $\hat{\theta}_{A}$. [The event where $\hat{\theta}_{LS,i}=0$ for
some $i$ has probability zero under the probability measure governing $u$.
Hence, it is inconsequential how we define $\hat{\theta}_{A}$ on this event;
for reasons of convenience, we shall adopt the convention that $\mu
_{n}^{2}/|\hat{\theta}_{LS,i}|=0$ if $\hat{\theta}_{LS,i}=0$. Furthermore, $%
\hat{\theta}_{A}$ is a measurable function of $Y$.] Note that \cite{zou06}
uses $\lambda _{n}=2n\mu _{n}^{2}$ as the tuning parameter. \cite{zou06}
also considers versions of the adaptive LASSO estimator for which $|\hat{%
\theta}_{LS,i}|$ in (\ref{alasso}) is replaced by $|\hat{\theta}%
_{LS,i}|^{\gamma }$. However, we shall exclusively concentrate on the
leading case $\gamma =1$. As pointed out in \cite{zou06}, the adaptive LASSO
is closely related to the nonnegative Garotte estimator of \cite{brei95}.


\section{Theoretical Analysis\label{theory}}


For the theoretical analysis in this section we shall make some simplifying
assumptions. First, we assume that $\sigma ^{2}$ is known, whence we may
assume without loss of generality that $\sigma ^{2}=1$. Second, we assume
orthogonal regressors, i.e., $X^{\prime }X$ is diagonal. The latter
assumption will be removed in the Monte Carlo study in Section \ref%
{simulation}. Orthogonal regressors occur in many important settings,
including wavelet regression or the analysis of variance. More specifically,
we shall assume $X^{\prime }X=nI_{k}$. In this case the minimization of (\ref%
{alasso}) is equivalent to separately minimizing 
\begin{equation}
n(\hat{\theta}_{LS,i}-\theta _{i})^{2}+2n\mu _{n}^{2}|\theta _{i}|/|\hat{%
\theta}_{LS,i}|  \label{alasso2}
\end{equation}%
for $i=1,\ldots ,k$. Since the estimators $\hat{\theta}_{LS,i}$ are
independent, so are the components of $\hat{\theta}_{A}$, provided $\mu _{n}$
is nonrandom which we shall assume for the theoretical analysis throughout
this section. To study the joint distribution of $\hat{\theta}_{A}$, it
hence suffices to study the distribution of the individual components.
Hence, we may assume\emph{\ without loss of generality }that\emph{\ }$\theta 
$ is scalar, i.e., $k=1$, for the rest of this section. In fact, as is
easily seen, there is then no loss of generality to even assume that $X$ is
just a column of $1$'s, i.e., we may then consider a simple Gaussian
location problem where $\hat{\theta}_{LS}=\bar{y}$, the arithmetic mean of
the independent and identically $N(\theta ,1)$-distributed observations $%
y_{1},\ldots ,y_{n}$. Under these assumptions, the minimization problem
defining the adaptive LASSO has an explicit solution of the form%
\begin{equation}
\hat{\theta}_{A}=\bar{y}(1-\mu _{n}^{2}/\bar{y}^{2})_{+}=\left\{ 
\begin{array}{cc}
0 & \text{if }\;|\bar{y}|\leq \mu _{n} \\ 
\bar{y}-\mu _{n}^{2}/\bar{y} & \text{if }\;|\bar{y}|>\mu _{n}.%
\end{array}%
\right.  \label{thetaA}
\end{equation}%
The explicit formula (\ref{thetaA}) also shows that in the location model
(and, more generally, in the diagonal regression model) the adaptive LASSO
and the nonnegative Garotte coincide, and thus the results in the present
section also apply to the latter estimator. In view of (\ref{thetaA}) we
also note that in the diagonal regression model the adaptive LASSO is
nothing else than a positive-part Stein estimator applied \emph{componentwise%
}. Of course, this is not in the spirit of Stein estimation.


\subsection{Model selection probabilities and tuning parameter \label{prob}}


The adaptive LASSO estimator $\hat{\theta}_{A}$ can be viewed as performing
a selection between the restricted model $M_{R}$ consisting only of the $%
N(0,1)$-distribution and the unrestricted model $M_{U}=\{N(\theta ,1):\theta
\in \mathbb{R}\}$ in an obvious way, i.e., $M_{R}$ is selected if $\hat{%
\theta}_{A}=0$ and $M_{U}$ is selected otherwise. We now study the model
selection probabilities, i.e., the probabilities that model $M_{U}$ or $%
M_{R} $, respectively, is selected. As these selection probabilities add up
to one, it suffices to consider one of them. The probability of selecting
the restricted model $M_{R}$ is given by%
\begin{eqnarray}
P_{n,\theta }(\hat{\theta}_{A}=0) &=&P_{n,\theta }(|\bar{y}|\leq \mu
_{n})=\Pr (|Z+n^{1/2}\theta |\leq n^{1/2}\mu _{n})  \notag \\
&=&\Phi (-n^{1/2}\theta +n^{1/2}\mu _{n})-\Phi (-n^{1/2}\theta -n^{1/2}\mu
_{n}),  \label{prob_MR}
\end{eqnarray}%
where $Z$ is a standard normal random variable with cumulative distribution
function (cdf) $\Phi $. We use $P_{n,\theta }$ to denote the probability
governing a sample of size $n$ when $\theta $ is the true parameter, and $%
\Pr $ to denote a generic probability measure.

In the following we shall always impose the condition that $\mu
_{n}\rightarrow 0$ for asymptotic considerations, which guarantees that the
probability of incorrectly selecting the restricted model $M_{R}$ (i.e.,
selecting $M_{R}$ if the true $\theta $ is non-zero) vanishes
asymptotically. Conversely, if this probability vanishes asymptotically for 
\emph{every} $\theta \neq 0$, then $\mu _{n}\rightarrow 0$ follows, hence
the condition $\mu _{n}\rightarrow 0$ is a basic one and without it the
estimator $\hat{\theta}_{A}$ does not seem to be of much interest.

Given the condition that $\mu _{n}\rightarrow 0$, two cases need to be
distinguished: (i) $n^{1/2}\mu _{n}\rightarrow \mathfrak{m}$, $0\leq 
\mathfrak{m}<\infty $ and (ii) $n^{1/2}\mu _{n}\rightarrow \infty $.%
\footnote{%
There is no loss in generality here in the sense that the general case where
only $\mu _{n}\rightarrow 0$ holds can always be reduced to case (i) or case
(ii) by passing to subsequences.} In case (i), the adaptive LASSO estimator
acts as a conservative model selection procedure, meaning that the
probability of selecting the larger model $M_{U}$ has a positive limit even
when $\theta =0$, whereas in case (ii), $\hat{\theta}_{A}$ acts as a
consistent model selection procedure, i.e., this probability vanishes in the
limit when $\theta =0$. This is immediately seen by inspection of (\ref%
{prob_MR}). In different guise, these facts have long been known, see \cite%
{bauetal88}. In his analysis of the adaptive LASSO estimator \cite{zou06}
assumes $n^{1/4}\mu _{n}\rightarrow 0$ and $n^{1/2}\mu _{n}\rightarrow
\infty $, hence he considers a subcase of case (ii). We shall discuss the
reason why \cite{zou06} imposes the stricter condition $n^{1/4}\mu
_{n}\rightarrow 0$ in Section \ref{asydistribs}.

The asymptotic behavior of the model selection probabilities discussed in
the preceding paragraph is of a \textquotedblleft
pointwise\textquotedblright\ asymptotic nature in the sense that the value
of $\theta $ is held fixed when $n\rightarrow \infty $. Since pointwise
asymptotic results often miss essential aspects of the finite-sample
behavior, we next present a \textquotedblleft moving
parameter\textquotedblright\ asymptotic analysis, i.e., we allow $\theta $
to vary with $n$ in the asymptotic analysis, which better reveals the
features of the problem in finite samples. Note that the following
proposition in particular shows that the convergence of the model selection
probability to its limit in a pointwise asymptotic analysis is \emph{not}
uniform in $\theta \in \mathbb{R}$ (in fact, it fails to be uniform in any
neighborhood of $\theta =0$).

\begin{proposition}
\label{selection_prob} Assume $\mu _{n}\rightarrow 0$ and $n^{1/2}\mu
_{n}\rightarrow \mathfrak{m}$ with $0\leq \mathfrak{m}\leq \infty $. \newline
(i) Assume $0\leq \mathfrak{m}<\infty $ (corresponding to conservative model
selection). Suppose that the true parameter $\theta _{n}\in \mathbb{R}$
satisfies $n^{1/2}\theta _{n}\rightarrow \nu \in \mathbb{R}\cup \{-\infty
,\infty \}$. Then 
\begin{equation*}
\lim_{n\rightarrow \infty }P_{n,\theta _{n}}(\hat{\theta}_{A}=0)=\Phi (-\nu +%
\mathfrak{m})-\Phi (-\nu -\mathfrak{m}).
\end{equation*}%
(ii) Assume $\mathfrak{m}=\infty $ (corresponding to consistent model
selection). Suppose $\theta _{n}\in \mathbb{R}$ satisfies $\theta _{n}/\mu
_{n}\rightarrow \zeta \in \mathbb{R}\cup \{-\infty ,\infty \}$. Then

\begin{enumerate}
\item $|\zeta| <1$ implies $\lim_{n \rightarrow \infty} P_{n,\theta_n}(\hat{%
\theta}_A=0) = 1$,

\item $|\zeta|=1$ and $n^{1/2}(\mu_n-\zeta \theta_n) \rightarrow r$ for some 
$r \in \mathbb{R} \cup \{-\infty,\infty\}$, implies $\lim_{n \rightarrow
\infty } P_{n,\theta_n}(\hat{\theta}_A=0)= \Phi(r)$,

\item $|\zeta |>1$ implies $\lim_{n\rightarrow \infty }P_{n,\theta _{n}}(%
\hat{\theta}_{A}=0)=0$.
\end{enumerate}
\end{proposition}

The proof of Proposition~\ref{selection_prob} is identical to the proof of
Proposition 1 in \cite{poelee07} and hence is omitted. The above proposition
in fact completely describes the large-sample behavior of the model
selection probability without \emph{any} conditions on the parameter $\theta 
$, in the sense that all possible accumulation points of the model selection
probability along \emph{arbitrary} sequences of $\theta _{n}$ can be
obtained in the following manner: Apply the result to subsequences and
observe that, by compactness of $\mathbb{R}\cup \{-\infty ,\infty \}$, we
can select from every subsequence a further subsequence such that all
relevant quantities such as $n^{1/2}\theta _{n}$, $\theta _{n}/\mu _{n}$, $%
n^{1/2}(\mu _{n}-\theta _{n})$, or $n^{1/2}(\mu _{n}+\theta _{n})$ converge
in $\mathbb{R}\cup \{-\infty ,\infty \}$ along this further subsequence.

In the case of conservative model selection, Proposition~\ref{selection_prob}
shows that the usual local alternative parameter sequences describe the
asymptotic behavior. In particular, if $\theta _{n}$ is local to $\theta =0$
in the sense that $\theta _{n}=\nu /n^{1/2}$, the local alternatives
parameter $\nu $ governs the limiting model selection probability.
Deviations of $\theta _{n}$ from $\theta =0$ of order $1/n^{1/2}$ are
detected with positive probability asymptotically and deviations of larger
order are detected with probability one asymptotically in this case. In the
consistent model selection case, however, a different picture emerges. Here,
Proposition~\ref{selection_prob} shows that local deviations of $\theta _{n}$
from $\theta =0$ that are of the order $1/n^{1/2}$ are not detected by the
model selection procedure at all!\footnote{%
For such deviations this also immediately follows from a contiguity argument.%
} In fact, even larger deviations from zero go asymptotically unnoticed by
the model selection procedure, namely as long as $\theta _{n}/\mu
_{n}\rightarrow \zeta $, $|\zeta |<1$. [Note that these larger deviations
would be picked up by a \emph{conservative} procedure with probability one
asymptotically.] This unpleasant consequence of model selection consistency
has a number of repercussions as we shall see later on. For a more detailed
discussion of these facts in the context of post-model-selection estimators
see \cite{leepoe05}.

The speed of convergence of the model selection probability to its limit in
part (i) of the proposition is governed by the slower of the convergence
speeds of $n^{1/2}\mu _{n}$ and $n^{1/2}\theta _{n}$. In part (ii), it is
exponential in $n^{1/2}\mu _{n}$ in cases 1 and 3, and is governed by the
convergence speed of $n^{1/2}\mu _{n}$ and $n^{1/2}(\mu _{n}-\zeta \theta
_{n})$ in case 2.


\subsection{Uniform consistency and uniform convergence rate of the adaptive
LASSO estimator\label{cons&ucons}}


It is easy to see that the natural condition $\mu _{n}\rightarrow 0$
discussed in the preceding section is in fact equivalent to consistency of $%
\hat{\theta}_{A}$ for $\theta $. Moreover, under this basic condition the
estimator is even uniformly consistent with a certain rate as we show next.

\begin{theorem}
\label{consist_ther} Assume that $\mu _{n}\rightarrow 0$. Then $\hat{\theta}%
_{A}$ is uniformly consistent for $\theta $, i.e., 
\begin{equation}
\lim_{n\rightarrow \infty }\sup_{\theta \in \mathbb{R}}P_{n,\theta }\left(
\left\vert \hat{\theta}_{A}-\theta \right\vert >\varepsilon \right) =0
\label{unif_consist}
\end{equation}%
for every $\varepsilon >0$. Furthermore, let $a_{n}=\min (n^{1/2},\mu
_{n}^{-1})$. Then, for every $\varepsilon >0$, there exists a (nonnegative)
real number $M$ such that 
\begin{equation}
\sup_{n\in \mathbb{N}}\sup_{\theta \in \mathbb{R}}P_{n,\theta }\left(
a_{n}\left\vert \hat{\theta}_{A}-\theta \right\vert >M\right) <\varepsilon
\label{unif_consist_rate}
\end{equation}%
holds. In particular, $\hat{\theta}_{A}$ is uniformly $a_{n}$-consistent.
\end{theorem}

For the case where the estimator $\hat{\theta}_{A}$ is tuned to perform
conservative model selection, the preceding theorem shows that these
estimators are uniformly $n^{1/2}$-consistent. In contrast, in case the
estimators are tuned to perform consistent model selection, the theorem only
guarantees uniform $\mu _{n}^{-1}$-consistency; that the estimator does
actually not converge faster than $\mu _{n}$ in a uniform sense will be
shown in Section \ref{asydistribs}.

\begin{remark}
\label{ezero} \normalfont In case $n^{1/2}\mu _{n}\rightarrow \mathfrak{m}$
with $\mathfrak{m}=0$, the adaptive LASSO estimator is uniformly
asymptotically equivalent to the unrestricted maximum likelihood estimator $%
\bar{y}$ in the sense that $\sup_{\theta \in \mathbb{R}}\;P_{n,\theta
}(n^{1/2}|\hat{\theta}_{A}-\bar{y}|>\varepsilon )\rightarrow 0$ for $%
n\rightarrow \infty $ and for every $\varepsilon >0$. Using (\ref{thetaA})
this follows easily from 
\begin{eqnarray*}
P_{n,\theta }(n^{1/2}|\hat{\theta}_{A}-\bar{y}|>\varepsilon ) &\leq &%
\boldsymbol{1}(n^{1/2}\mu _{n}>\varepsilon )+P_{n,\theta }(n^{1/2}\mu
_{n}^{2}/|\bar{y}|>\varepsilon ,|\bar{y}|>\mu _{n}) \\
&\leq &2\cdot \boldsymbol{1}(n^{1/2}\mu _{n}>\varepsilon )\rightarrow 0.
\end{eqnarray*}
\end{remark}


\subsection{The distribution of the adaptive LASSO\label{Section3}}



\subsubsection{Finite-sample distributions \label{finite}}


We now derive the finite-sample distribution of $n^{1/2}(\hat{\theta}%
_{A}-\theta )$. For purpose of comparison we note the obvious fact that the
distribution of the unrestricted maximum likelihood estimator $\hat{\theta}%
_{U}=\bar{y}$ (corresponding to model $M_{U}$) as well as the distribution
of the restricted maximum likelihood estimator $\hat{\theta}_{R}\equiv 0$
(corresponding to model $M_{R})$ are normal. More precisely, $n^{1/2}(\hat{%
\theta}_{U}-\theta )$ is $N(0,1)$-distributed and $n^{1/2}(\hat{\theta}%
_{R}-\theta )$ is $N(-n^{1/2}\theta ,0)$-distributed, where the singular
normal distribution is to be interpreted as pointmass at $-n^{1/2}\theta $.
[The latter is simply an instance of the fact that in case $k>1$ the
restricted estimator has a singular normal distribution concentrated on the
subspace defined by the zero restrictions.]

The finite-sample distribution $F_{A,n,\theta }$ of $n^{1/2}(\hat{\theta}%
_{A}-\theta )$ is given by%
\begin{eqnarray*}
P_{n,\theta }(n^{1/2}(\hat{\theta}_{A}-\theta )\leq x) &=&P_{n,\theta
}(n^{1/2}(\hat{\theta}_{A}-\theta )\leq x,\hat{\theta}_{A}=0) \\
&+&P_{n,\theta }(n^{1/2}(\hat{\theta}_{A}-\theta )\leq x,\hat{\theta}_{A}>0)
\\
&+&P_{n,\theta }(n^{1/2}(\hat{\theta}_{A}-\theta )\leq x,\hat{\theta}_{A}<0)
\\
&=&A+B+C.
\end{eqnarray*}%
By (\ref{prob_MR}) we clearly have 
\begin{equation*}
A\;=\;\boldsymbol{1}(-n^{1/2}\theta \leq x)\;\left\{ \Phi (-n^{1/2}\theta
+n^{1/2}\mu _{n})-\Phi (-n^{1/2}\theta -n^{1/2}\mu _{n})\right\} .
\end{equation*}%
Furthermore, using expression (\ref{thetaA}) we find that 
\begin{eqnarray*}
B &=&P_{n,\theta }(n^{1/2}(\bar{y}-\mu _{n}^{2}/\bar{y}-\theta )\leq x,\bar{y%
}>\mu _{n}) \\
&=&P_{n,\theta }(n^{1/2}(\bar{y}^{2}-\mu _{n}^{2}-\theta \bar{y})\leq \bar{y}%
x,\bar{y}>\mu _{n}) \\
&=&\Pr (Z^{2}+n^{1/2}\theta Z-n\mu _{n}^{2}\leq Zx+n^{1/2}\theta
x,Z>-n^{1/2}\theta +n^{1/2}\mu _{n}) \\
&=&\Pr (Z^{2}+(n^{1/2}\theta -x)Z-(n\mu _{n}^{2}+n^{1/2}\theta x)\leq 0,\
Z>-n^{1/2}\theta +n^{1/2}\mu _{n}),
\end{eqnarray*}%
where $Z$ follows a standard normal distribution. The quadratic form in $Z$
is convex and hence is less than or equal to zero precisely between the
zeroes of the equation 
\begin{equation*}
z^{2}+(n^{1/2}\theta -x)z-(n\mu _{n}^{2}+n^{1/2}\theta x)=0.
\end{equation*}%
The solutions $z_{n,\theta }^{(1)}(x)$ and $z_{n,\theta }^{(2)}(x)$ of this
equation with $z_{n,\theta }^{(1)}(x)\leq z_{n,\theta }^{(2)}(x)$ are given
by%
\begin{equation}
-(n^{1/2}\theta -x)/2\pm \sqrt{((n^{1/2}\theta +x)/2)^{2}+n\mu _{n}^{2}}.
\label{zeros_equ}
\end{equation}%
Note that the expression under the root in (\ref{zeros_equ}) is always
positive, so that 
\begin{equation*}
B\;=\;P_{n,\theta }\left( z_{n,\theta }^{(1)}(x)\leq Z\leq z_{n,\theta
}^{(2)}(x),\;Z>-n^{1/2}\theta +n^{1/2}\mu _{n}\right) .
\end{equation*}%
Observe that $z_{n,\theta }^{(1)}(x)\leq -n^{1/2}\theta +n^{1/2}\mu _{n}$
always holds and that $-n^{1/2}\theta +n^{1/2}\mu _{n}\leq z_{n,\theta
}^{(2)}(x)$ is equivalent to $n^{1/2}\theta +x\geq 0$, so that we can write 
\begin{equation*}
B=\boldsymbol{1}(n^{1/2}\theta +x\geq 0)\;\left\{ \Phi \left( z_{n,\theta
}^{(2)}(x)\right) -\Phi (-n^{1/2}\theta +n^{1/2}\mu _{n})\right\} .
\end{equation*}%
The term $C$ can be treated in a similar fashion to arrive at 
\begin{equation*}
C\;=\;\boldsymbol{1}(n^{1/2}\theta +x\geq 0)\;\Phi (-n^{1/2}\theta
-n^{1/2}\mu _{n})+\boldsymbol{1}(n^{1/2}\theta +x<0)\;\Phi \left(
z_{n,\theta }^{(1)}(x)\right) .
\end{equation*}%
Adding up $A$, $B$ and $C$, we now obtain the finite-sample distribution
function of $n^{1/2}(\hat{\theta}_{A}-\theta )$ as%
\begin{equation}
F_{A,n,\theta }(x)=\boldsymbol{1}(n^{1/2}\theta +x\geq 0)\;\Phi \left(
z_{n,\theta }^{(2)}(x)\right) +\boldsymbol{1}(n^{1/2}\theta +x<0)\;\Phi
\left( z_{n,\theta }^{(1)}(x)\right) .  \label{fscdf}
\end{equation}%
It follows that the distribution of $n^{1/2}(\hat{\theta}_{A}-\theta )$
consists of an atomic part given by 
\begin{equation}
\left\{ \Phi (n^{1/2}(-\theta +\mu _{n}))-\Phi (n^{1/2}(-\theta -\mu
_{n}))\right\} \delta _{-n^{1/2}\theta },  \label{singpt}
\end{equation}%
where $\delta _{z}$ represents pointmass at the point $z$, and an absolutely
continuous part that has a Lebesgue density given by%
\begin{eqnarray}
&&0.5\times \left\{ \boldsymbol{1}(n^{1/2}\theta +x>0)\;\phi \left(
z_{n,\theta }^{(2)}(x)\right) \;(1+t_{n,\theta }(x))\;+\;\right.  \notag \\
&&\left. \boldsymbol{1}(n^{1/2}\theta +x<0)\;\phi \left( z_{n,\theta
}^{(1)}(x)\right) \;(1-t_{n,\theta }(x))\right\} ,  \label{abscontpt}
\end{eqnarray}%
where $t_{n,\theta }(x)=0.5(n^{1/2}\theta +x)/\left( ((n^{1/2}\theta
+x)/2)^{2}+n\mu _{n}^{2}\right) ^{1/2}$. Figure~1 illustrates the shape of
the finite-sample distribution of $n^{1/2}(\hat{\theta}_{A}-\theta )$.
Obviously, the distribution is highly non-normal.

\begin{figure}[ht]
\begin{center}
\includegraphics{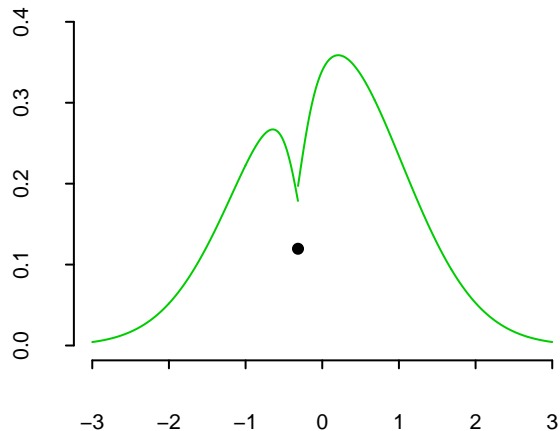}
\end{center}
\caption{Distribution of $n^{1/2}(\hat{\protect\theta}_{A}-\protect\theta )$
for $n=10$, $\protect\theta =0.1$, $\protect\mu_{n}=0.05$. The plot shows
the density of the absolutely continuous part (\protect\ref{abscontpt}), as
well as the total mass of the atomic part (\protect\ref{singpt}) located at $%
-n^{1/2}\protect\theta =-0.32$.}
\end{figure}


\subsubsection{Asymptotic distributions \label{asydistribs}}


We next obtain the asymptotic distributions of $\hat{\theta}_{A}$ under
general \textquotedblleft moving parameter\textquotedblright\ asymptotics
(i.e., asymptotics where the true parameter can depend on sample size),
since -- as already noted earlier -- considering only fixed-parameter
asymptotics may paint a very misleading picture of the behavior of the
estimator. In fact, the results given below amount to a complete description
of all possible accumulation points of the finite-sample distribution, cf.
Remarks \ref{r1}. Not surprisingly, the results in the conservative model
selection case are different from the ones in the consistent model selection
case.


\paragraph{Conservative case}


The large-sample behavior of the distribution $F_{A,n,\theta _{n}}$ of $%
n^{1/2}(\hat{\theta}_{A}-\theta _{n})$ for the case when the estimator is
tuned to perform conservative model selection is characterized in the
following theorem.

\begin{theorem}
\label{asymp_conserv} Assume $\mu _{n}\rightarrow 0$ and $n^{1/2}\mu
_{n}\rightarrow \mathfrak{m}$, $0\leq \mathfrak{m}<\infty $. Suppose the
true parameter $\theta _{n}\in \mathbb{R}$ satisfies $n^{1/2}\theta
_{n}\rightarrow \nu \in \mathbb{R}\cup \{-\infty ,\infty \}$. Then, for $\nu
\in \mathbb{R}$, $F_{A,n,\theta _{n}}$ converges weakly to the distribution%
\begin{multline*}
\boldsymbol{1}(x+\nu \geq 0)\Phi \left( -{\frac{\nu -x}{2}}+\sqrt{({\frac{%
\nu +x}{2}})^{2}+\mathfrak{m}^{2}}\right) \\
+\boldsymbol{1}(x+\nu <0)\Phi \left( -{\frac{\nu -x}{2}}-\sqrt{({\frac{\nu +x%
}{2}})^{2}+\mathfrak{m}^{2}}\right) .
\end{multline*}%
If $|\nu |=\infty $, then $F_{A,n,\theta _{n}}$ converges weakly to $\Phi $,
i.e., to a standard normal distribution.
\end{theorem}

The fixed-parameter asymptotic distribution can be obtained from Theorem~\ref%
{asymp_conserv} by setting $\theta _{n}\equiv \theta $: For $\theta =0$, we
get 
\begin{equation*}
\boldsymbol{1}(x\geq 0)\;\Phi (x/2+\sqrt{(x/2)^{2}+\mathfrak{m}^{2}})+%
\boldsymbol{1}(x<0)\;\Phi (x/2-\sqrt{(x/2)^{2}+\mathfrak{m}^{2}}),
\end{equation*}%
which coincides with the finite-sample distribution in (\ref{fscdf}) except
for replacing $n^{1/2}\mu _{n}$ with its limit $\mathfrak{m}$. However, for $%
\theta \neq 0$, the resulting fixed-parameter asymptotic distribution is a
standard normal distribution which clearly misrepresents the actual
distribution (\ref{fscdf}). This disagreement is most pronounced in the
statistically interesting case where $\theta $ is close to, but not equal
to, zero (e.g., $\theta \sim n^{-1/2}$). In contrast, the distribution given
in Theorem \ref{asymp_conserv} much better captures the behavior of the
finite-sample distribution also in this case because it coincides with the
finite-sample distribution (\ref{fscdf}) except for the fact that $%
n^{1/2}\mu _{n}$ and $n^{1/2}\theta _{n}$ have settled down to their
limiting values.


\paragraph{Consistent case}


In this subsection we consider the case where the tuning parameter $\mu _{n}$
is chosen so that $\hat{\theta}_{A}$ performs consistent model selection,
i.e. $\mu _{n}\rightarrow 0$ and $n^{1/2}\mu _{n}\rightarrow \infty $.

\begin{theorem}
\label{asymp_consist} Assume that $\mu _{n}\rightarrow 0$ and $n^{1/2}\mu
_{n}\rightarrow \infty $. Suppose the true parameter $\theta _{n}\in \mathbb{%
R}$ satisfies $\theta _{n}/\mu _{n}\rightarrow \zeta $ for some $\zeta \in 
\mathbb{R}\cup \{-\infty ,\infty \}$.

\begin{enumerate}
\item If $\zeta =0$ and $n^{1/2}\theta _{n}\rightarrow \nu \in \mathbb{R}$,
then $F_{A,n,\theta _{n}}$ converges weakly to the cdf $\boldsymbol{1}(\cdot
\geq -\nu )$.

\item The total mass $F_{A,n,\theta _{n}}$ escapes to either $\infty $ or $%
-\infty $ for the following cases: If $-\infty <\zeta <0$, or if $\zeta =0$
and $n^{1/2}\theta _{n}\rightarrow -\infty $, or if $\zeta =-\infty $ and $%
n^{1/2}\mu _{n}^{2}/\theta _{n}\rightarrow -\infty $, then $F_{A,n,\theta
_{n}}(x)\rightarrow 0$ for every $x\in \mathbb{R}$. If $0<\zeta <\infty $,
or if $\zeta =0$ and $n^{1/2}\theta _{n}\rightarrow \infty $, or if $\zeta
=\infty $ and $n^{1/2}\mu _{n}^{2}/\theta _{n}\rightarrow \infty $, then $%
F_{A,n,\theta _{n}}(x)\rightarrow 1$ for every $x\in \mathbb{R}$.

\item If $\left\vert \zeta \right\vert =\infty $ and $n^{1/2}\mu
_{n}^{2}/\theta _{n}\rightarrow r\in \mathbb{R}$, then $F_{A,n,\theta _{n}}$
converges weakly to the cdf $\Phi (\cdot +r)$.
\end{enumerate}
\end{theorem}

The fixed-parameter asymptotic behavior of the adaptive LASSO estimator is
obtained from Theorem~\ref{asymp_consist} by setting $\theta _{n}\equiv
\theta $: For $\theta =0$, the asymptotic distribution reduces to point-mass
at $0$, which coincides with the asymptotic distribution of the restricted
maximum likelihood estimator. In the case of $\theta \neq 0$, the asymptotic
distribution is $\Phi (x+\rho /\theta )$ provided $n^{1/2}\mu
_{n}^{2}\rightarrow \rho $ (with the obvious interpretation if $\left\vert
\rho \right\vert =\infty $). That is, it is a \emph{shifted }version of the
asymptotic distribution of the unrestricted maximum likelihood estimator
(the shift being infinitely large if $\left\vert \rho \right\vert =\infty $%
). Observe that (for $\left\vert \rho \right\vert <\infty $) the shift gets
larger as $\left\vert \theta \right\vert $, $\left\vert \theta \right\vert
\neq 0$, gets smaller. The `oracle' property in the sense of \cite{zou06} is
hence satisfied if and only if $\rho =0$, that is, if the tuning parameter
additionally also satisfies $n^{1/4}\mu _{n}\rightarrow 0$. This is
precisely the condition imposed in Theorem~2 in \cite{zou06} which
establishes the `oracle' property. [Note that $n^{1/4}\mu _{n}\rightarrow 0$
translates into the assumption $\lambda _{n}/n^{1/2}\rightarrow 0$ in
Theorem~2 in \cite{zou06}.] If $n^{1/2}\mu _{n}^{2}\rightarrow \rho \neq 0$,
the adaptive LASSO estimator provides an example of an estimator that
performs consistent model selection, but does not satisfy the `oracle'
property in the sense that for $\theta \neq 0$ its asymptotic distribution
does not coincide with the asymptotic distribution of the unrestricted
maximum likelihood estimator.

In any case, the `oracle' property, which is guaranteed under the additional
requirement $n^{1/4}\mu _{n}\rightarrow 0$, carries little statistical
meaning: Imposing the additional condition $n^{1/4}\mu _{n}\rightarrow 0$
still allows all three cases in Theorem~\ref{asymp_consist} above to occur,
showing that -- notwithstanding the validity of the `oracle' property --
non-normal limiting distributions arise under a moving-parameter asymptotic
framework. These latter distributions are in better agreement with the
features exhibited by the finite-sample distribution (\ref{fscdf}), whereas
the `oracle' property always predicts a normal limiting distribution (a
singular one in case $\theta =0$), showing that it does not capture
essential features of the finite-sample distribution. In particular, the
preceding theorem shows that the estimator is not uniformly $n^{1/2}$%
-consistent as the sequence of finite-sample distributions of $n^{1/2}(\hat{%
\theta}_{A}-\theta _{n})$ is stochastically unbounded in some cases arising
in Theorem~\ref{asymp_consist}. All this goes to show that the `oracle'
property, which is based on the pointwise asymptotic distribution only,
paints a highly misleading picture of the behavior of the adaptive LASSO
estimator and should not be taken at face value. See also Remark \ref{r4}.

It transpires from Theorem \ref{asymp_consist} that $F_{A,n,\theta _{n}}$
converges weakly to the singular normal distribution $N(0,0)$ if $\theta
_{n}=0$ for all $n$, and to the standard normal $N(0,1)$ if $\theta _{n}$
satisfies $\left\vert \theta _{n}\right\vert /(n^{1/2}\mu
_{n}^{2})\rightarrow \infty $. Hence, if one, for example, allows as the
parameter space for $\theta $ only the set $\Theta _{n}=\left\{ \theta \in 
\mathbb{R}:\theta =0\text{ or }\left\vert \theta \right\vert >b_{n}\right\} $
where $b_{n}>0$ satisfies $b_{n}/(n^{1/2}\mu _{n}^{2})\rightarrow \infty $,
then the convergence of $F_{A,n,\theta _{n}}$ to the limiting distributions $%
N(0,0)$ and $N(0,1)$, respectively, is uniform over $\Theta _{n}$, i.e., the
`oracle' property holds uniformly over $\Theta _{n}$. Does this line of
reasoning restore the credibility of the `oracle' property? We do not think
so for the following reasons: The choice of $\Theta _{n}$ as the parameter
space is highly artificial, depends on sample size as well as on the tuning
parameter (and hence on the estimation procedure). Furthermore, in case $%
\Theta _{n}$ is adopted as the parameter space, the `forbidden' set $\mathbb{%
R-}\Theta _{n}$ will always have a diameter that is of order larger than $%
n^{-1/2}$; in fact, it will always contain elements $\theta _{n}\neq 0$ such
that $\theta _{n}$ would be correctly classified as non-zero with
probability converging to one by the adaptive LASSO procedure used, i.e., $%
P_{n,\theta _{n}}(\hat{\theta}_{A}\neq 0)\rightarrow 1$ (to see this note
that $\mathbb{R-}\Theta _{n}$ contains elements $\theta _{n}$ satisfying $%
\mathbb{\theta }_{n}/\mu _{n}\rightarrow \zeta $ with $\left\vert \zeta
\right\vert >1$ and use Proposition 1). This shows that adopting $\Theta _{n}
$ as the parameter space rules out values of $\theta $ that are
substantially different from zero, and not only values of $\theta $ that are
difficult to statistically distinguish from zero; consequently the
`forbidden' set is sizable. Summarizing, there appears to be little reason
why $\Theta _{n}$ would be a natural choice of parameter space, especially
in the context of model selection where interest naturally focusses on the
neighborhood of zero. We therefore believe that using $\Theta _{n}$ as the
parameter space is hardly supported by statistical reasoning but is more
reflective of a search for conditions that are favorable to the `oracle'
property.

As mentioned above, Theorem \ref{asymp_consist} shows, in particular, that $%
\hat{\theta}_{A}$ is not uniformly $n^{1/2}$-consistent. This prompts the
question of the behavior of the distribution of $c_{n}(\hat{\theta}%
_{A}-\theta _{n})$ under a sequence of norming constants $c_{n}$ that are $%
o(n^{1/2})$. Inspection of the proof of Theorem \ref{asymp_consist} reveals
that the stochastic unboundedness phenomenon persists if $c_{n}$ is $%
o(n^{1/2})$ but is of order larger than $\mu _{n}^{-1}$. For $c_{n}=O(\mu
_{n}^{-1})$, we always have stochastic boundedness by Theorem \ref%
{consist_ther}. Hence, the uniform convergence rate of $\hat{\theta}_{A}$ is
seen to be $\mu _{n}$ which is slower than $n^{-1/2}$. The precise limit
distributions of the estimator under the scaling $c_{n}\sim \mu _{n}^{-1}$
is obtained in the next theorem. [The case $c_{n}=o(\mu _{n}^{-1})$ is
trivial since then these limits are always pointmass at zero in view of
Theorem~\ref{consist_ther}.\footnote{%
There is no loss in generality here in the sense that the general case where 
$c_{n}=O(\mu _{n}^{-1})$ holds can -- by passing to subsequences -- always
be reduced to the cases where $c_{n}\sim \mu _{n}^{-1}$ or $c_{n}=o(\mu
_{n}^{-1})$ holds.}] A consequence of the next theorem is that with such a
scaling the \emph{pointwise }limiting distributions always degenerate to
pointmass at zero. This points to something of a dilemma with the adaptive
LASSO estimator when tuned to perform consistent model selection: If we
scale the estimator by $\mu _{n}^{-1}$, i.e., by the `right' uniform rate,
the pointwise limiting distributions degenerate to pointmass at zero. If we
scale the estimator by $n^{1/2}$, which is the `right' pointwise rate (at
least if $n^{1/4}\mu _{n}\rightarrow 0$), then we end up with stochastically
unbounded sequences of distributions under a moving parameter asymptotic
framework (for certain sequences $\theta _{n}$).

Let $G_{A,n,\theta }$ stand for the finite-sample distribution of $\mu
_{n}^{-1}(\hat{\theta}_{A}-\theta )$ under $P_{n,\theta }$. Clearly, $%
G_{A,n,\theta }(x)=F_{A,n,\theta }(n^{1/2}\mu _{n}x)$. The limits of this
distribution under `moving parameter' asymptotics are given in the
subsequent theorem. It turns out that the limiting distributions are always
pointmasses, however, not always located at zero.

\begin{theorem}
\label{asymp_rescaled} Assume that $\mu _{n}\rightarrow 0$, $n^{1/2}\mu
_{n}\rightarrow \infty $, and that $\theta _{n}/\mu _{n}\rightarrow \zeta $
for some $\zeta \in \mathbb{R}\cup \{-\infty ,\infty \}$.

\begin{enumerate}
\item If $|\zeta |<1$, then $G_{A,n,\theta _{n}}$ converges weakly to the
cdf $\boldsymbol{1}(\cdot \geq -\zeta )$.

\item If $1\leq |\zeta |<\infty $, then $G_{A,n,\theta _{n}}$ converges
weakly the cdf $\boldsymbol{1}(\cdot \geq -1/\zeta )$.

\item If $|\zeta |=\infty $, then $G_{A,n,\theta _{n}}$ converges weakly to
the cdf $\boldsymbol{1}(\cdot \geq 0)$.
\end{enumerate}
\end{theorem}

\subsubsection{Some Remarks}

\begin{remark}
\label{r1}\normalfont Theorems~\ref{asymp_conserv} and \ref{asymp_consist}
actually completely describe all accumulation points of the finite-sample
distribution of $n^{1/2}(\hat{\theta}_{A}-\theta _{n})$ without \emph{any}
condition on the sequence of parameters $\theta _{n}$. To see this, just
apply the theorems to subsequences and note that by compactness of $\mathbb{R%
}\cup \{-\infty ,\infty \}$ we can select from every subsequence a further
subsequence such that the relevant quantities like $n^{1/2}\theta _{n}$, $%
\theta _{n}/\mu _{n}$, and $n^{1/2}\mu _{n}^{2}/\theta _{n}$ converge in $%
\mathbb{R}\cup \{-\infty ,\infty \}$ along this further subsequence. A
similar comment also applies to Theorem \ref{asymp_rescaled}.
\end{remark}

\begin{remark}
\normalfont As a point of interest we note that the full complexity of the
possible limiting distributions in Theorems~\ref{asymp_conserv}, \ref%
{asymp_consist}, and \ref{asymp_rescaled} already arises if we restrict the
sequences $\theta _{n}$ to a bounded neighborhood of zero. Hence, the
phenomena described by the above theorems are of a local nature, and are not
tied in any way to the unboundedness of the parameter space.
\end{remark}

\begin{remark}
\normalfont In case the estimator is tuned to perform consistent model
selection, it is mainly the behavior of $\theta _{n}/\mu _{n}$ that governs
the form of the limiting distributions in Theorems \ref{asymp_consist} and %
\ref{asymp_rescaled}. Note that $\theta _{n}/\mu _{n}$ is of smaller order
than $n^{1/2}\theta _{n}$ because $n^{1/2}\mu _{n}\rightarrow \infty $ in
the consistent case. Hence, an analysis relying only on the classical local
asymptotics based on perturbations of $\theta $ of the order of $n^{-1/2}$
does not properly reveal all possible limits of the finite-sample
distributions in that case. [This is in contrast to the conservative case,
where classical local asymptotics reveal all possible limit distributions.]
\end{remark}

\begin{remark}
\label{r4}\normalfont The mathematical reason for the failure of the
pointwise asymptotic distributions to capture the behavior of the
finite-sample distributions well is that the convergence of the latter to
the former is not uniform in the underlying parameter $\theta \in \mathbb{R}$%
. See \cite{leepoe03,leepoe05} for more discussion in the context of
post-model-selection estimators.
\end{remark}

\begin{remark}
\normalfont The theoretical analysis has been restricted to the case of
orthogonal regressors. In the case of correlated regressors we can expect to
see similar phenomena (e.g., non-normality of finite-sample cdfs,
non-uniformity problems, etc.), although details will be different. Evidence
for this is provided by the simulation study presented in Section \ref%
{simulation}, by corresponding theoretical results for a class of
post-model-selection estimators (\cite{leepoe03,leepoe06a,leepoe08a}) in the
correlated regressor case as well as by general results on estimators
possessing the sparsity property (\cite{leepoe08b,poe07}).
\end{remark}


\subsection{Impossibility results for estimating the distribution of the
adaptive LASSO\label{imp}}


Since the cdf $F_{A,n,\theta }$ of $n^{1/2}(\hat{\theta}_{A}-\theta )$
depends on the unknown parameter, as shown in Section~\ref{finite}, one
might be interested in estimating this cdf. We show that this is an
intrinsically difficult estimation problem in the sense that the cdf cannot
be estimated in a uniformly consistent fashion. In the following, we provide
large-sample results that cover both consistent and conservative choices of
the tuning parameter, as well as finite-sample results that hold for any
choice of tuning parameter. For related results in different contexts see 
\cite{leepoe06a,leepoe06b,leepoe08b,poe06,poelee07}.

It is straightforward to construct consistent estimators for the
distribution $F_{A,n,\theta }$ of the (centered and scaled) estimator $\hat{%
\theta}_{A}$. One popular choice is to use subsampling or the $m$ out of $n$
bootstrap with $m/n\rightarrow 0$. Another possibility is to use the
pointwise large-sample limit distributions derived in Section~\ref%
{asydistribs} together with a properly chosen pre-test of the hypothesis $%
\theta =0$ versus $\theta \neq 0$. Because the pointwise large-sample limit
distribution takes only two different functional forms depending on whether $%
\theta =0$ or $\theta \neq 0$, one can perform a pre-test that rejects the
hypothesis $\theta =0$ in case $|\bar{y}|>n^{-1/4}$, say, and estimate the
finite-sample distribution by that large-sample limit formula that
corresponds to the outcome of the pre-test;\footnote{%
In the conservative case, the asymptotic distribution can also depend on $%
\mathfrak{m}$ which is then to be replaced by $n^{1/2}\mu _{n}$.} the test's
critical value $n^{-1/4}$ ensures that the correct large-sample limit
formula is selected with probability approaching one as sample size
increases. However, as we show next, any consistent estimator of the cdf $%
F_{A,n,\theta }$ is necessarily badly behaved in a worst-case sense.

\begin{theorem}
\label{imp1} Let $\mu _{n}$ be a sequence of tuning parameters such that $%
\mu _{n}\rightarrow 0$ and $n^{1/2}\mu _{n}\rightarrow \mathfrak{m}$ with $%
0\leq \mathfrak{m}\leq \infty $. Let $t\in \mathbb{R}$ be arbitrary. Then
every consistent estimator $\hat{F}_{n}(t)$ of $F_{A,n,\theta }(t)$
satisfies 
\begin{equation*}
\lim_{n\rightarrow \infty }\sup_{|\theta |<c/n^{1/2}}P_{n,\theta }\left(
\left\vert \hat{F}_{n}(t)-F_{A,n,\theta }(t)\right\vert \;>\;\varepsilon
\right) \;\;=\;\;1
\end{equation*}%
for each $\varepsilon <(\Phi (t+\mathfrak{m})-\Phi (t-\mathfrak{m}))/2$ and
each $c>\left\vert t\right\vert $. In particular, no uniformly consistent
estimator for $F_{A,n,\theta }(t)$ exists.
\end{theorem}

We stress that the above result also applies to any kind of bootstrap- or
subsampling-based estimator of the cdf $F_{A,n,\theta }$ whatsoever, since
the results in \cite{leepoe06b} on which the proof of Theorem~\ref{imp1}
rests apply to arbitrary randomized estimators, cf. Lemma 3.6 in \cite%
{leepoe06b}. The same applies to Theorems~\ref{imp2} and \ref{imp3} that
follow.

Loosely speaking, Theorem~\ref{imp1} states that any consistent estimator
for the cdf $F_{A,n,\theta }$ suffers from an unavoidable worst-case error
of at least $\varepsilon $ with $\varepsilon <(\Phi (t+\mathfrak{m})-\Phi (t-%
\mathfrak{m}))/2$. The error range, i.e., $(\Phi (t+\mathfrak{m})-\Phi (t-%
\mathfrak{m}))/2$, is governed by the limit $\mathfrak{m}=\lim_{n}n^{1/2}\mu
_{n}$. In case the estimator is tuned to be consistent, i.e., in case $%
\mathfrak{m}=\infty $, the error range equals $1/2$, and the phenomenon is
most pronounced. If the estimator is tuned to be conservative so that $%
\mathfrak{m}<\infty $, the error range is less than $1/2$ but can still be
substantial. Only in case $\mathfrak{m}=0$ the error range equals zero, and
the condition $\varepsilon <(\Phi (t+\mathfrak{m})-\Phi (t-\mathfrak{m}))/2$
in Theorem~\ref{imp1} leads to a trivial conclusion. This is, however, not
surprising as then the resulting estimator is uniformly asymptotically
equivalent to the unrestricted maximum likelihood estimator $\bar{y}$, cf.
Remark~\ref{ezero}.

A similar non-uniformity phenomenon as described in Theorem~\ref{imp1} for
consistent estimators $\hat{F}_{n}(t)$ also occurs for not necessarily
consistent estimators. For such arbitrary estimators we find in the
following that the phenomenon can be somewhat less pronounced, in the sense
that the lower bound is now $1/2$ instead of $1$, cf. (\ref{imp2rel2})
below. The following theorem gives a large-sample limit result that
parallels Theorem~\ref{imp1}, as well as a finite-sample result, both for
arbitrary (and not necessarily consistent) estimators of the cdf.

\begin{theorem}
\label{imp2}Let $0<\mu _{n}<\infty $ and let $t\in \mathbb{R}$ be arbitrary.
Then every estimator $\hat{F}_{n}(t)$ of $F_{A,n,\theta }(t)$ satisfies 
\begin{equation}
\sup_{|\theta |<c/n^{1/2}}P_{n,\theta }\left( \left\vert \hat{F}%
_{n}(t)-F_{A,n,\theta }(t)\right\vert \;>\;\varepsilon \right) \;\;\geq \;\;%
\frac{1}{2}  \label{imp2rel1}
\end{equation}%
for each $\varepsilon <(\Phi (t+n^{1/2}\mu _{n})-\Phi (t-n^{1/2}\mu _{n}))/2$%
, for each $c>|t|$, and for each fixed sample size $n$. If $\mu _{n}$
satisfies $\mu _{n}\rightarrow 0$ and $n^{1/2}\mu _{n}\rightarrow \mathfrak{m%
}$ as $n\rightarrow \infty $ with $0\leq \mathfrak{m}\leq \infty $, we thus
have 
\begin{equation}
\liminf_{n\rightarrow \infty }\inf_{\hat{F}_{n}(t)}\sup_{|\theta
|<c/n^{1/2}}P_{n,\theta }\left( \left\vert \hat{F}_{n}(t)-F_{A,n,\theta
}(t)\right\vert \;>\;\varepsilon \right) \;\;\geq \;\;\frac{1}{2}
\label{imp2rel2}
\end{equation}%
for each $\varepsilon <(\Phi (t+\mathfrak{m})-\Phi (t-\mathfrak{m}))/2$ and
for each $c>|t|$, where the infimum in (\ref{imp2rel2}) extends over all
estimators $\hat{F}_{n}(t)$.
\end{theorem}

The finite-sample statement in Theorem~\ref{imp2} clearly reveals how the
estimability of the cdf of the estimator depends on the tuning parameter $%
\mu _{n}$: A larger value of $\mu _{n}$, which results in a `more sparse'
estimator in view of (\ref{prob_MR}), directly corresponds to a larger range 
$(\Phi (t+n^{1/2}\mu _{n})-\Phi (t-n^{1/2}\mu _{n}))/2$ for the error $%
\varepsilon $ within which any estimator $\hat{F}_{n}(t)$ performs poorly in
the sense of (\ref{imp2rel1}). In large samples, the limit $\mathfrak{m}%
=\lim_{n\rightarrow \infty }n^{1/2}\mu _{n}$ takes the role of $n^{1/2}\mu
_{n}$.

An impossibility result paralleling Theorem~\ref{imp2} for the cdf $%
G_{A,n,\theta }(t)$ of $\mu _{n}^{-1}(\hat{\theta}_{A}-\theta )$ is given
next.

\begin{theorem}
\label{imp3}Let $0<\mu _{n}<\infty $ and let $t\in \mathbb{R}$ be arbitrary.
Then every estimator $\hat{G}_{n}(t)$ of $G_{A,n,\theta }(t)$ satisfies 
\begin{equation}
\sup_{|\theta |<c\mu _{n}}P_{n,\theta }\left( \left\vert \hat{G}%
_{n}(t)-G_{A,n,\theta }(t)\right\vert \;>\;\varepsilon \right) \quad \geq
\quad \frac{1}{2}  \label{imp3rel1}
\end{equation}%
for each $\varepsilon <(\Phi (n^{1/2}\mu _{n}(t+1))-\Phi (n^{1/2}\mu
_{n}(t-1)))/2$, for each $c>|t|$, and for each fixed sample size $n$. If $%
\mu _{n}$ satisfies $\mu _{n}\rightarrow 0$ and $n^{1/2}\mu _{n}\rightarrow
\infty $ as $n\rightarrow \infty $, we thus have for each $c>|t|$ 
\begin{equation}
\liminf_{n\rightarrow \infty }\inf_{\hat{G}_{n}(t)}\sup_{|\theta |<c\mu
_{n}}P_{n,\theta }\left( \left\vert \hat{G}_{n}(t)-G_{A,n,\theta
}(t)\right\vert \;>\;\varepsilon \right) \quad \geq \quad \frac{1}{2}
\label{imp3rel2}
\end{equation}%
for each $\varepsilon <1/2$ if $\left\vert t\right\vert <1$ and for each $%
\varepsilon <1/4$ if $\left\vert t\right\vert =1$, where the infimum in (\ref%
{imp3rel2}) extends over all estimators $\hat{G}_{n}(t)$.
\end{theorem}

This result shows, in particular, that no uniformly consistent estimator
exists for $G_{A,n,\theta }(t)$ in case $\left\vert t\right\vert \leq 1$
(not even over compact subsets of $\mathbb{R}$ containing the origin). In
view of Theorem~\ref{asymp_rescaled}, we see that for $t>1$ we have $%
\sup_{\theta \in \mathbb{R}}\left\vert G_{A,n,\theta }(t)-1\right\vert
\rightarrow 0$ as $n\rightarrow \infty $, hence $\hat{G}_{n}(t)=1$ is
trivially a uniformly consistent estimator in this case. Similarly, for $%
t<-1 $ we have $\sup_{\theta \in \mathbb{R}}\left\vert G_{A,n,\theta
}(t)\right\vert \rightarrow 0$ as $n\rightarrow \infty $, hence $\hat{G}%
_{n}(t)=0$ is trivially a uniformly consistent estimator in this case.


\section{Some Monte Carlo Results\label{simulation}}


We provide simulation results for the finite-sample distribution of the
adaptive LASSO estimator in the case of \emph{non-orthogonal }regressors to
complement our theoretical findings for the orthogonal case. We present our
results by showing the marginal distribution for each component of the
centered and scaled estimator. Not surprisingly, the graphs exhibit the same
highly non-normal features of the corresponding finite-sample distribution
of the estimator derived in Section \ref{Section3} for the case of
orthogonal regressors.

The simulations were carried out the following way. We consider 1000
repetitions of $n$ simulated data points from the model (\ref{model}) with $%
\sigma ^{2}=1$ and $X$ such that $X^{\prime }X=n\Omega $ with $\Omega
_{ij}=0.5^{|i-j|}$ for $i,j=1,\dots ,k$. More concretely, $X$ was
partitioned into $d=n/k$ blocks of size $k\times k$ (where $d$ is assumed to
be integer) and each of these blocks was set equal to $k^{1/2}L$, with $%
LL^{\prime }=\Omega $, the Cholesky factorization of $\Omega $. We used $k=4$
regressors and various values of the true parameter $\theta $ given by $%
\theta =(3,1.5,\gamma n^{-1/2},\gamma n^{-1/2})^{\prime }$ where $\gamma
=0,1,2$. This model with $\theta =(3,1.5,0,0)^{\prime }$ (i.e., $\gamma =0$)
is a downsized version of a model considered in Monte Carlo studies in \cite%
{tib96}, \cite{fanli01}, and \cite{zou06}. For apparent reasons it is of
interest to investigate the performance of the estimator not only at a
single parameter value, but also at other (neighboring) points in the
parameter space. The cases with $\gamma \neq 0$, represent the statistically
interesting case where some components of the true parameter value are close
to but not equal to zero.

For each simulation, we computed the adaptive LASSO estimator $\hat{\theta}%
_{A}$ using the LARS package of \cite{efretal04} in R. Each component of the
estimator was centered and scaled, i.e., $C_{jj}^{-1/2}(\hat{\theta}%
_{A,j}-\theta _{j})$ was computed, where $C=(n\Omega )^{-1}$. The tuning
parameter $\mu _{n}$ was chosen in two different ways. In the first case, it
was set to the fixed value of $\mu _{n}=n^{-1/3}$, a choice that corresponds
to consistent model selection and additionally satisfies the condition $%
n^{1/4}\mu _{n}\rightarrow 0$ required in \cite{zou06} to obtain the
'oracle' property. In the second case, in each simulation the tuning
parameter was selected to minimize a mean-squared prediction error obtained
through $K$-fold cross-validation (which can be computed using the LARS
package, in our case with $K=10$).

The results for both choices of the tuning parameters, for $n=100$, and $%
\gamma =0,1,2$ are shown in Figures 2-7 below. For each component of the
estimator, the discrete component of the distribution corresponding to the
zero values of the $j$-th component of the estimator $\hat{\theta}_{A,j}$
(appearing at $-C_{jj}^{-1/2}\theta _{j}$ for the centered and scaled
estimator) is represented by a dot drawn at the height of the corresponding
relative frequency. The histogram formed from the remaining values of $%
C_{jj}^{-1/2}(\hat{\theta}_{A,j}-\theta _{j})$ was then smoothed by the
kernel smoother available in R, resulting in the curves representing the
density of the absolutely continuous part of the finite-sample distribution
of $C_{jj}^{-1/2}(\hat{\theta}_{A,j}-\theta _{j})$. Naturally, in these
plots the density was rescaled by the appropriate relative frequency of the
estimator not being equal to zero.

We first discuss the case where the tuning parameter is set at the fixed
value $\mu _{n}=n^{-1/3}$. For $\gamma =0$, i.e., the case where the last
two components of the true parameter are identically zero, Figure 2 shows
that the adaptive LASSO estimator finds the zero components in $\theta
=(3,1.5,0,0)^{\prime }$ with probability close to one (i.e., the
distributions of $C_{jj}^{-1/2}(\hat{\theta}_{A,j}-\theta _{j})$, $j=3,4$,
practically coincide with pointmass at $0$). Furthermore, the distributions
of the first two components seem to somewhat resemble normality. The outcome
in this case is hence roughly in line with what the 'oracle' property
predicts. This is due to the fact that the components of $\theta $ are
either zero or large (note that $C_{jj}^{-1/2}\theta _{j}$ is approximately
equal to $26$ and $12$, respectively, for $j=1,2$). The results are quite
different for the cases $\gamma =1$ and $\gamma =2$ (Figures 3 and 4), which
represent the case where some of the components of the parameter vector $%
\theta $ are large and some are different from zero but small (note that $%
C_{33}^{-1/2}\theta _{3}\approx 0.77\gamma $ and $C_{44}^{-1/2}\theta
_{4}\approx 0.87\gamma $). In both cases the distributions of $C_{jj}^{-1/2}(%
\hat{\theta}_{A,j}-\theta _{j})$, $j=3,4$, are a mixture of an atomic part
and an absolutely continuous part, both shifted to the left of the origin.
Furthermore, the absolutely continuous part appears to be highly non-normal.
This is perfectly in line with the theoretical results obtained in Section %
\ref{Section3}. It once again demonstrates that the 'oracle' property gives
a misleading impression of the actual performance of the estimator.

\begin{figure}[tbp]
\begin{center}
\includegraphics[angle=-90, width=0.9\textwidth]{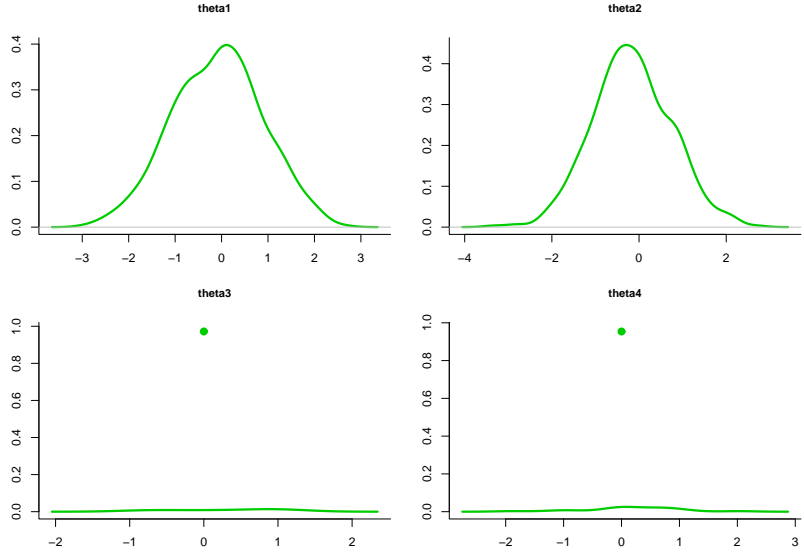}
\end{center}
\caption{Marginal distributions of the scaled and centered adaptive LASSO
estimator for $n=100$, $\protect\gamma =0$, i.e., $\protect\theta %
=(3,1.5,0,0)^{\prime }$, and $\protect\mu_{n}=n^{-1/3}=0.22$.}
\end{figure}

\begin{figure}[tbp]
\begin{center}
\includegraphics[angle=-90, width=0.9\textwidth]{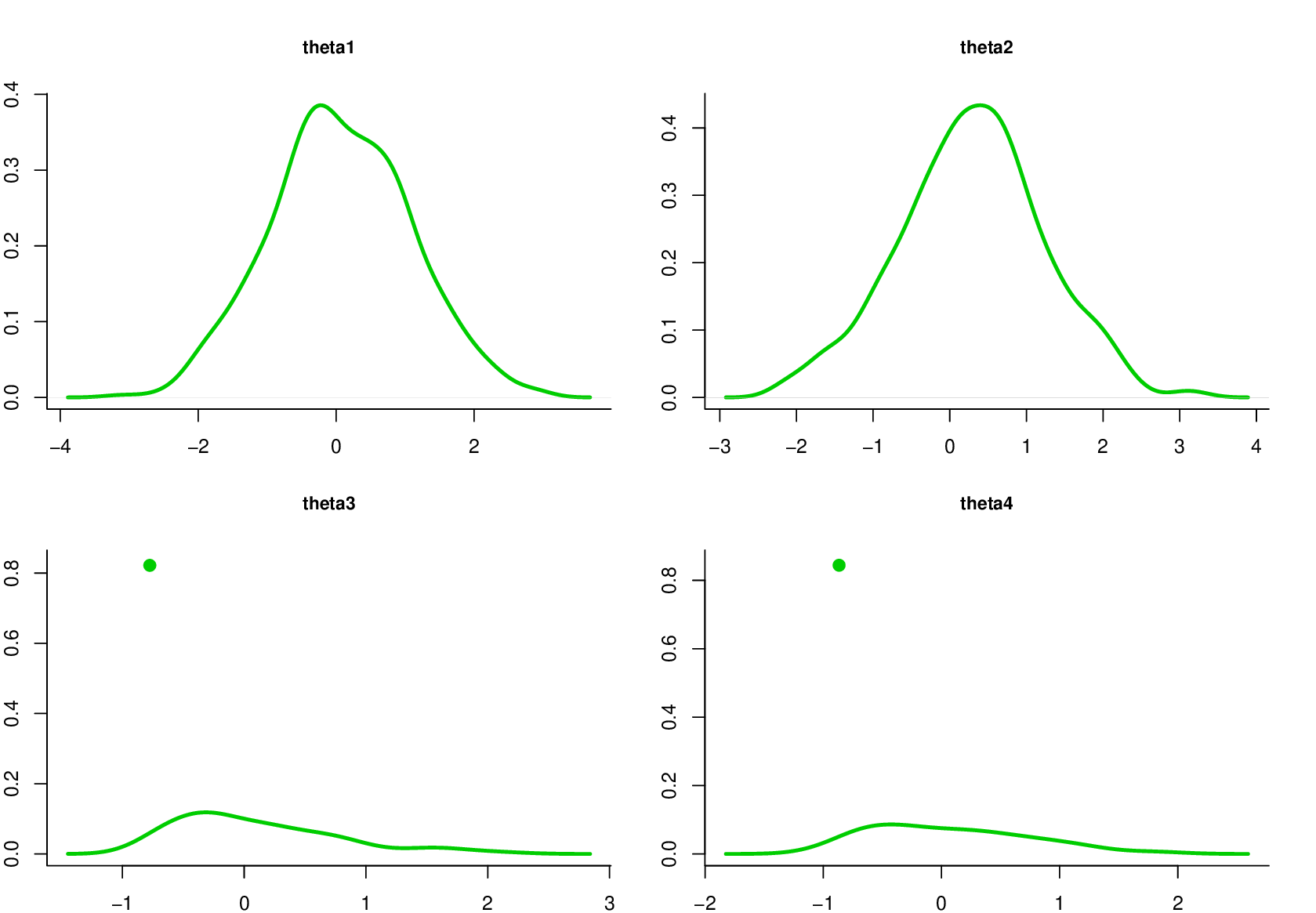}
\end{center}
\caption{Marginal distributions of the scaled and centered adaptive LASSO
estimator for $n=100$, $\protect\gamma =1$, i.e., $\protect\theta %
=(3,1.5,0.1,0.1)^{\prime }$, and $\protect\mu_{n}=n^{-1/3}=0.22$.}
\end{figure}

\begin{figure}[tbp]
\begin{center}
\includegraphics[angle=-90, width=0.9\textwidth]{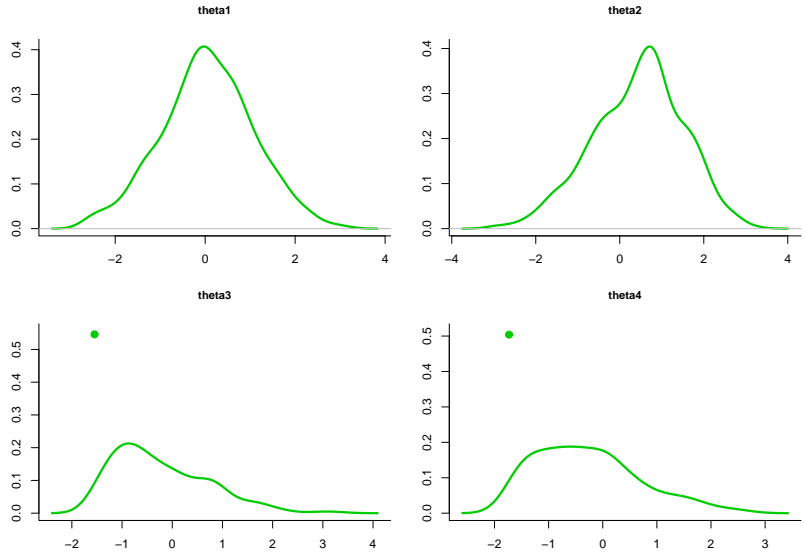}
\end{center}
\caption{Marginal distributions of the scaled and centered adaptive LASSO
estimator for $n=100$, $\protect\gamma =2$, i.e., $\protect\theta %
=(3,1.5,0.2,0.2)^{\prime }$, and $\protect\mu_{n}=n^{-1/3}=0.22$.}
\end{figure}

In the case where the tuning parameter is chosen by cross-validation, a
similar picture emerges, except for the fact that in case $\gamma =0$ the
adaptive LASSO estimator now finds the zero component less frequently, cf.
Figure 5. [In fact, the probability of finding a zero value of $\hat{\theta}%
_{A,j}$ for $j=3,4$ is smaller in the cross-validated case regardless of the
value of $\gamma $ considered.] The reason for this is that the tuning
parameters obtained through cross-validation were typically found to be
smaller than $n^{-1/3}$, resulting in an estimator $\hat{\theta}_{A}$ that
acts more like a conservative rather than a consistent model selection
procedure. [This is in line with theoretical results in \cite{leliwa06}, see
also \cite{leepoe08a}.] In agreement with the theoretical results in Section %
\ref{Section3}, the absolutely continuous components of the distributions of 
$C_{jj}^{-1/2}(\hat{\theta}_{A,j}-\theta _{j})$ are now typically highly
non-normal, especially for $j=3,4$, cf. Figures 5-7. [Note that
cross-validation leads to a data-depending tuning parameter $\mu _{n}$, \ a
situation that is strictly speaking not covered by the theoretical results.]

\begin{figure}[tbp]
\begin{center}
\includegraphics[angle=-90, width=0.9\textwidth]{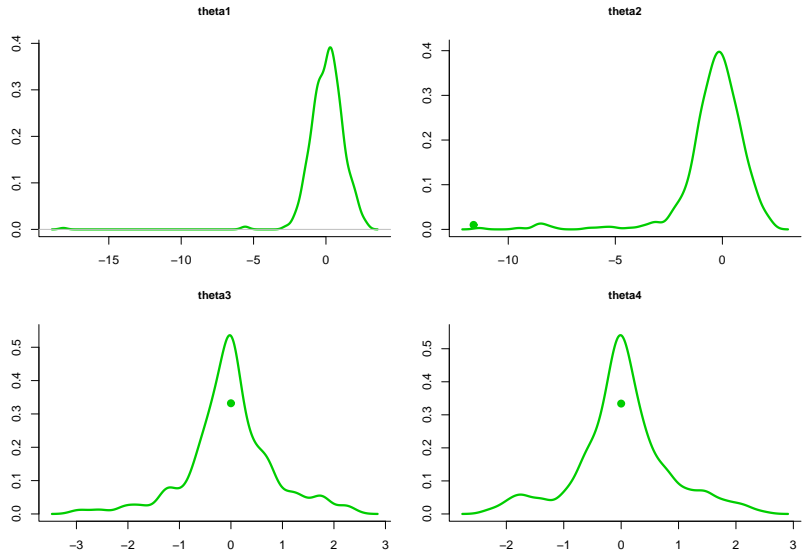}
\end{center}
\caption{Marginal distributions of the scaled and centered adaptive LASSO
estimator for $n=100$, $\protect\gamma =0$, i.e., $\protect\theta %
=(3,1.5,0,0)^{\prime }$, and $\protect\mu_{n}$ chosen by cross-validation.}
\end{figure}

\begin{figure}[tbp]
\begin{center}
\includegraphics[angle=-90, width=0.9\textwidth]{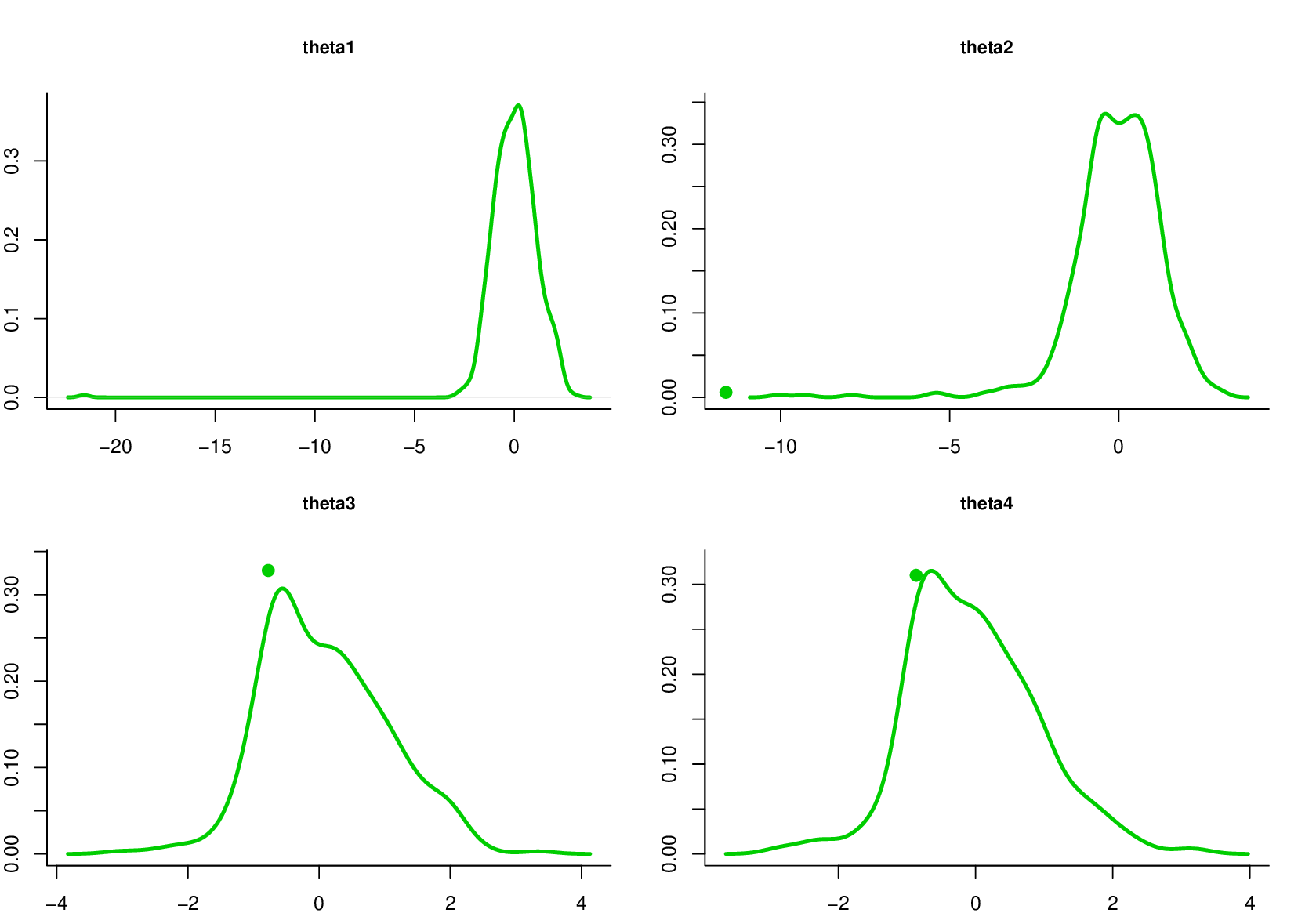}
\end{center}
\caption{Marginal distributions of the scaled and centered adaptive LASSO
estimator for $n=100$, $\protect\gamma =1$, i.e., $\protect\theta %
=(3,1.5,0.1,0.1)^{\prime }$, and $\protect\mu_{n}$ chosen by
cross-validation.}
\end{figure}

\begin{figure}[tbp]
\begin{center}
\includegraphics[angle=-90, width=0.9\textwidth]{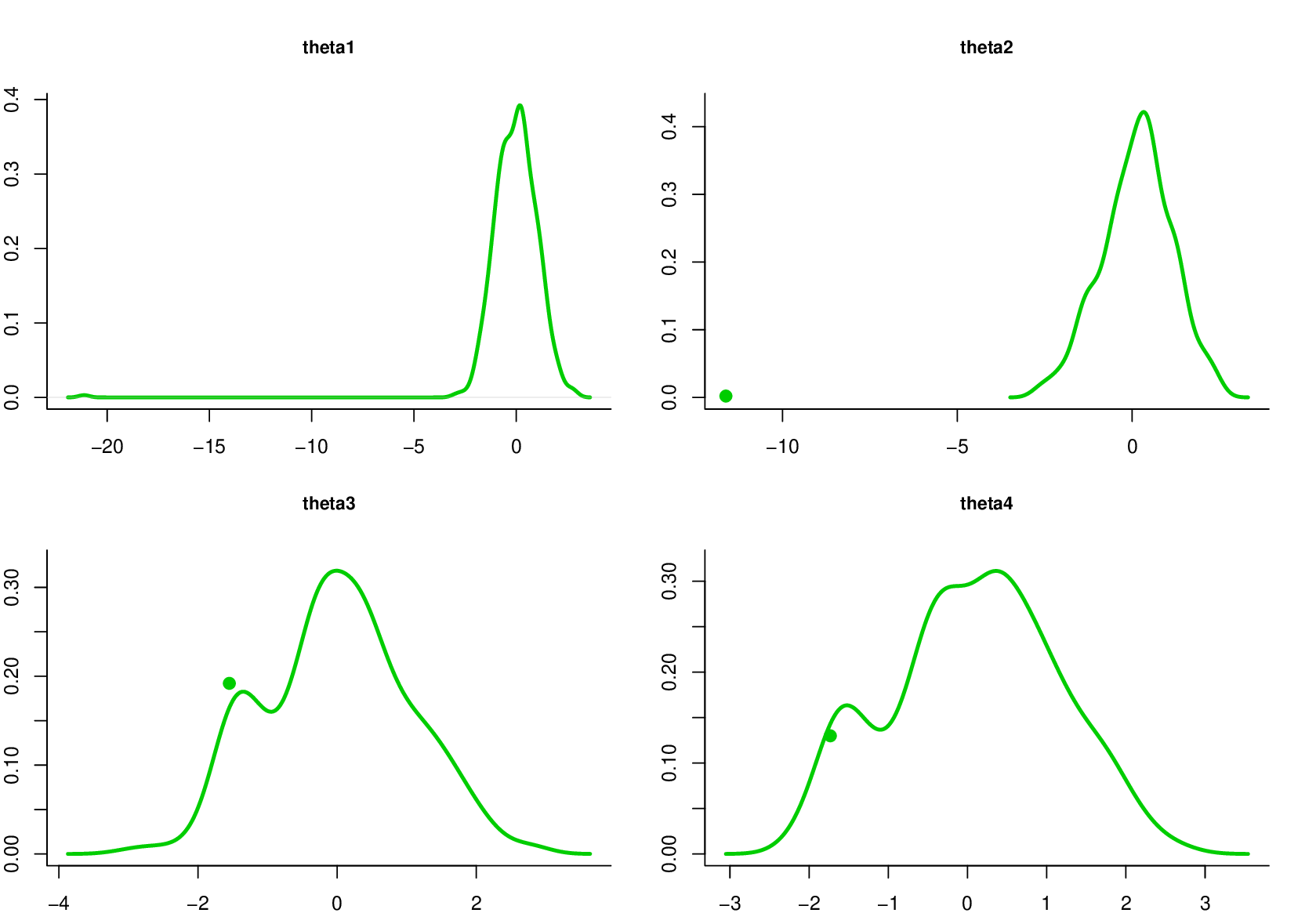}
\end{center}
\caption{Marginal distributions of the scaled and centered adaptive LASSO
estimator for $n=100$, $\protect\gamma =2$, i.e., $\protect\theta %
=(3,1.5,0.2,0.2)^{\prime }$, and $\protect\mu_{n}$ chosen by
cross-validation.}
\end{figure}

We have also experimented with other values of $\theta $ such as $\theta
=(3,1.5,\gamma n^{-1/2},0)^{\prime }$ or $\theta =(3,1.5,0,\gamma
n^{-1/2}\,)^{\prime }$, other values of $\gamma $ and other sample sizes
such as $n=60$ or $200$. The results were found to be qualitatively the same.


\section{Conclusion \label{conl}}


We have studied the distribution of the adaptive LASSO estimator, a
penalized least squares estimator introduced in \cite{zou06}, in
finite-samples as well as in the large-sample limit. The theoretical study
assumes an orthogonal regression model. The finite-sample distribution was
found to be a mixture of a singular normal distribution and an absolutely
continuous distribution, which is non-normal. The large-sample limit of the
distributions depends on the choice of the estimator's tuning parameter, and
we can distinguish two cases:

In the first case the tuning is such that the estimator acts as a
conservative model selector. In this case, the adaptive LASSO estimator is
found to be uniformly $n^{1/2}$-consistent. We also show that
fixed-parameter asymptotics (where the true parameter remains fixed while
sample size increases) only partially reflect the actual behavior of the
distribution whereas \textquotedblleft moving-parameter\textquotedblright\
asymptotics (where the true parameter may depend on sample size) gives a
more accurate picture. The moving-parameter analysis shows that the
distribution may be highly non-normal irrespective of sample size, in
particular, in the statistically interesting case where the true parameter
is close (in an appropriate sense) to a lower-dimensional submodel. This
also implies that the finite-sample phenomena that we have observed can
occur at any sample size.

In the second case, where the estimator is tuned to perform consistent model
selection, again fixed-parameter asymptotics do not capture the whole range
of large-sample phenomena that can occur. With `moving parameter'
asymptotics, we have shown that the distribution of these estimators can
again be highly non-normal, even in large samples. In addition, we have
found that the observed finite-sample phenomena not only can persist but
actually can be more pronounced for larger sample sizes. For example, the
distribution of the estimator (properly centered and scaled by $n^{1/2}$)
can diverge in the sense that all its mass escapes to either $+\infty $ or $%
-\infty $. In fact, we have established that the uniform convergence rate of
the adaptive LASSO estimator is slower than $n^{-1/2}$ in the consistent
model selection case. These findings are especially important as the
adaptive LASSO estimator has been shown in \cite{zou06} to possess an
'oracle' property (under an additional assumption on the tuning parameter),
which promises a convergence rate of $n^{-1/2}$ and a normal distribution in
large samples. However, the 'oracle' property is based on a fixed-parameter
asymptotic argument which, as our results show, gives highly misleading
results.

The findings mentioned above are based on a theoretical analysis (Section %
\ref{theory}) of the adaptive LASSO estimator in an orthogonal linear
regression model. The orthogonality restriction is removed in the Monte
Carlo analysis in Section \ref{simulation}. The results from this simulation
study confirm the theoretical results.

Finally, we have studied the problem of estimating the cdf of the (centered
and scaled) adaptive LASSO estimator. We have shown that this cdf cannot be
estimated in a uniformly consistent fashion, even though pointwise
consistent estimators can be constructed with relative ease.

We would like to stress that our results should not be read as a
condemnation of the adaptive LASSO estimator, but as a warning that the
distributional properties of this estimator are quite intricate and complex.

\appendix{}

\section{Appendix}


\textbf{Proof of Theorem \ref{consist_ther}: }Since (\ref{unif_consist_rate}%
) implies (\ref{unif_consist}), it suffices to prove the former. For this,
it is instructive to write $\hat{\theta}_{A}$ in terms of the
hard-thresholding estimator $\hat{\theta}_{H}$ as defined in \cite{poelee07}
(with $\eta _{n}=\mu _{n}$) by observing that 
\begin{equation*}
\hat{\theta}_{A}=\hat{\theta}_{H}-\limfunc{sign}(\hat{\theta}_{H})\mu
_{n}^{2}/|\bar{y}|.
\end{equation*}%
Here $\limfunc{sign}(x)=-1,0,1$ depending on whether $x<0,=0,>0$. Since $%
\hat{\theta}_{H}$ satisfies (\ref{unif_consist_rate}) as is shown in
Theorem~ 2 in \cite{poelee07}, it suffices to consider 
\begin{eqnarray*}
\sup_{\theta \in \mathbb{R}}P_{n,\theta }(a_{n}|\hat{\theta}_{H}-\hat{\theta}%
_{A}|>M) &=&\sup_{\theta \in \mathbb{R}}P_{n,\theta }(a_{n}\mu _{n}^{2}/|%
\bar{y}|>M,\hat{\theta}_{H}\neq 0) \\
&=&\sup_{\theta \in \mathbb{R}}P_{n,\theta }(a_{n}\mu _{n}^{2}/|\bar{y}|>M,|%
\bar{y}|>\mu _{n}) \\
&\leq &\boldsymbol{1}(a_{n}\mu _{n}>M).
\end{eqnarray*}%
Since $a_{n}\mu _{n}\leq 1$, the right-hand side in the above expression
equals zero for any $M>1$.\hfill $\blacksquare $

\begin{proposition}
\label{zeros_limit} Let $\theta _{n}\in \mathbb{R}$ and $0<\mu _{n}<\infty $%
. If $\theta _{n}/\mu _{n}\rightarrow -\infty $ and $n^{1/2}\theta
_{n}\rightarrow -\infty $, then $z_{n,\theta _{n}}^{(1)}(x)-x\sim n^{1/2}\mu
_{n}^{2}/\theta _{n}$ as $n\rightarrow \infty $ for every $x\in \mathbb{R}$.
If $\theta _{n}/\mu _{n}\rightarrow \infty $ and $n^{1/2}\theta
_{n}\rightarrow \infty $, then $z_{n,\theta _{n}}^{(2)}(x)-x\sim n^{1/2}\mu
_{n}^{2}/\theta _{n}$ for every $x\in \mathbb{R}$.
\end{proposition}

\begin{proof}
We prove the first claim. We can write 
\begin{eqnarray*}
z_{n,\theta _{n}}^{(1)}(x)-x &=&-(n^{1/2}\theta _{n}+x)/2-\sqrt{%
((n^{1/2}\theta _{n}+x)/2)^{2}+n\mu _{n}^{2}} \\
&=&n^{1/2}\alpha _{n}(x)\left\{ -1+\sqrt{1+(\mu _{n}/\alpha _{n}(x))^{2}}%
\right\}
\end{eqnarray*}%
with $n^{1/2}\alpha _{n}(x)=(n^{1/2}\theta _{n}+x)/2$ where the last
equality holds for large $n$ since $n^{1/2}\alpha _{n}(x)<0$ eventually.
Through an expansion of $\sqrt{1+z}$ about zero, we obtain 
\begin{eqnarray*}
z_{n,\theta _{n}}^{(1)}(x)-x &=&n^{1/2}(\mu _{n}^{2}/\alpha _{n}(x))(1+\bar{z%
}_{n})^{-1/2}/2 \\
&=&(n^{1/2}\mu _{n}^{2}/\theta _{n})(1+x/(n^{1/2}\theta _{n}))^{-1}(1+\bar{z}%
_{n})^{-1/2},
\end{eqnarray*}%
with $0\leq \bar{z}_{n}\leq (\mu _{n}/\alpha _{n}(x))^{2}$. Note that $\mu
_{n}/\alpha _{n}(x)=2(\mu _{n}/\theta _{n})(1+x/(n^{1/2}\theta
_{n}))^{-1}\rightarrow 0$, and hence $\bar{z}_{n}\rightarrow 0$ holds. The
claim now follows. The second claim is proved analogously.
\end{proof}

\textbf{Proof of Theorem \ref{asymp_conserv}: }We derive the corresponding
asymptotic distributions by studying the limit behavior of (\ref{fscdf})
with $\theta $ replaced by $\theta _{n}$. If $\nu \in \mathbb{R}$ the result
immediately follows, since $F_{A,n,\theta _{n}}(x)$ converges to the limit
given above for every $x\neq -\nu $ as a consequence of (\ref{zeros_equ})
and $n^{1/2}\theta _{n}\rightarrow \nu $. For the case $\nu =\infty $, note
that the indicator function of the first term in (\ref{fscdf}) goes to $1$
for every $x\in \mathbb{R}$, whereas the second one goes to $0$.
Furthermore, we clearly have $\theta _{n}/\mu _{n}\rightarrow \infty $ since 
$0\leq \mathfrak{m}<\infty $ holds. Therefore we can apply Proposition~\ref%
{zeros_limit} to find that $z_{n,\theta _{n}}^{(2)}(x)\rightarrow x$ since $%
n^{1/2}\mu _{n}^{2}/\theta _{n}=n^{1/2}\mu _{n}(\mu _{n}/\theta
_{n})\rightarrow \mathfrak{m}\cdot 0=0$. This implies that $F_{A,n,\theta
}(x)\rightarrow \Phi (x)$ for all $x\in \mathbb{R}$ in case $\nu =\infty $.
A similar argument can be made to prove the claim for $\nu =-\infty $.\hfill 
$\blacksquare $

\textbf{Proof of Theorem \ref{asymp_consist}: }If $\left\vert \zeta
\right\vert <1$, Proposition~\ref{selection_prob} shows that the total mass
of the atomic part (\ref{singpt}) of the distribution $F_{A,n,\theta _{n}}$
goes to $1$; furthermore, the location of the atomic part, i.e., $%
-n^{1/2}\theta _{n}$, then converges to $-\nu \in \mathbb{R}$ or to $\pm
\infty $. This proves the theorem in case $\left\vert \zeta \right\vert <1$.
We prove the remaining cases by inspecting the limit behavior of (\ref{fscdf}%
), again with $\theta _{n}$ replacing $\theta $. To derive the limits for $%
1\leq \left\vert \zeta \right\vert \leq \infty $, note that $n^{1/2}\theta
_{n}\rightarrow \limfunc{sign}(\zeta )\,\infty $, so that by assessing the
limit of the indicator functions in (\ref{fscdf}), it can easily be seen
that $F_{A,n,\theta _{n}}(x)$ converges to the limit of $\Phi (z_{n,\theta
_{n}}^{(2)}(x))$ for $\zeta >0$ and to the limit of $\Phi (z_{n,\theta
_{n}}^{(1)}(x))$ for $\zeta <0$. Elementary calculations show that $%
z_{n,\theta _{n}}^{(2)}(x)\rightarrow \infty $ for $1\leq \zeta <\infty $
and that $z_{n,\theta _{n}}^{(1)}(x)\rightarrow -\infty $ for $-\infty
<\zeta \leq -1$. As a consequence of Proposition~\ref{zeros_limit}, also $%
z_{n,\theta _{n}}^{(2)}(x)\rightarrow \infty $ if $\zeta =\infty $ and $%
n^{1/2}\mu _{n}^{2}/\theta _{n}\rightarrow \infty $; similarly, $z_{n,\theta
_{n}}^{(1)}(x)\rightarrow -\infty $ if $\zeta =-\infty $ and $n^{1/2}\mu
_{n}^{2}/\theta _{n}\rightarrow -\infty $. This then proves the remaining
cases in part 2. Under the assumptions of part 3, an application of
Proposition~\ref{zeros_limit} gives that $z_{n,\theta
_{n}}^{(2)}(x)\rightarrow x+r$ if $\zeta =\infty $ and that $z_{n,\theta
_{n}}^{(1)}(x)\rightarrow x+r$ if $\zeta =-\infty $, which then proves part
3.\hfill $\blacksquare $

\textbf{Proof of Theorem \ref{asymp_rescaled}: }To prove part 1, observe
that Proposition~\ref{selection_prob} implies $\lim_{n\rightarrow \infty
}P_{n,\theta _{n}}(\hat{\theta}_{A}=0)=1$ for $|\zeta |<1$. This entails%
\begin{align*}
\lim_{n\rightarrow \infty }P_{n,\theta _{n}}(\mu _{n}^{-1}(\hat{\theta}%
_{A}-\theta _{n})\leq x)& =\lim_{n\rightarrow \infty }P_{n,\theta _{n}}(\mu
_{n}^{-1}(\hat{\theta}_{A}-\theta _{n})\leq x,\,\hat{\theta}_{A}=0) \\
& =\lim_{n\rightarrow \infty }\boldsymbol{1}(-\theta _{n}/\mu _{n}\leq x)=%
\boldsymbol{1}(x\geq -\zeta )
\end{align*}%
for $x\neq -\zeta $, which establishes part 1. Next, observe that%
\begin{align}
G_{A,n,\theta _{n}}(x)& =  \notag \\
\boldsymbol{1}(\theta _{n}/\mu _{n}& +x\geq 0)\Phi (w_{n,\theta
_{n}}^{(2)}(x))+\boldsymbol{1}(\theta _{n}/\mu _{n}+x<0)\Phi (w_{n,\theta
_{n}}^{(1)}(x))  \label{fscdfrescaled}
\end{align}%
where $w_{n,\theta _{n}}^{(1)}(x)$ and $w_{n,\theta _{n}}^{(2)}(x)$ with $%
w_{n,\theta _{n}}^{(1)}(x)\leq w_{n,\theta _{n}}^{(2)}(x)$ are given by%
\begin{equation}
n^{1/2}\mu _{n}\left\{ (-\theta _{n}/\mu _{n}+x)\pm \sqrt{(\theta _{n}/\mu
_{n}+x)^{2}+4}\right\} /2.  \label{w}
\end{equation}%
Under the conditions of part 2, the first indicator function in (\ref%
{fscdfrescaled}) tends to $1$ for $x>-\zeta $ and to $0$ for $x<-\zeta $.
Consequently, $G_{A,n,\theta _{n}}(x)$ converges to $\lim_{n\rightarrow
\infty }\Phi (w_{n,\theta _{n}}^{(2)}(x))$ if $x>-\zeta $, and to $%
\lim_{n\rightarrow \infty }\Phi (w_{n,\theta _{n}}^{(1)}(x))$ if $x<-\zeta $
(provided the limits exist). Elementary calculations show that for $\zeta
\geq 1$ we have $w_{n,\theta _{n}}^{(1)}(x)\rightarrow -\infty $ for all $%
x\in \mathbb{R}$, $w_{n,\theta _{n}}^{(2)}(x)\rightarrow -\infty $ for $%
x<-1/\zeta $, and $w_{n,\theta _{n}}^{(2)}(x)\rightarrow \infty $ for $%
x>-1/\zeta $. For $\zeta \leq -1$ we obtain $w_{n,\theta
_{n}}^{(1)}(x)\rightarrow -\infty $ for $x<-1/\zeta $, $w_{n,\theta
_{n}}^{(1)}(x)\rightarrow \infty $ for $x>-1/\zeta $, and $w_{n,\theta
_{n}}^{(2)}(x)\rightarrow \infty $ for all $x\in \mathbb{R}$. Consequently,
for $x\neq -\zeta $, we find $G_{A,n,\theta _{n}}(x)\rightarrow 0$ for $%
x<-1/\zeta $ and $G_{A,n,\theta _{n}}(x)\rightarrow 1$ for $x>-1/\zeta $. If 
$\left\vert \zeta \right\vert =1$, the result in part 2 follows. If $%
\left\vert \zeta \right\vert >1$, convergence of $G_{A,n,\theta _{n}}(-\zeta
)$ to the proper limit follows from monotonicity of $G_{A,n,\theta _{n}}$
and the fact that $x=-\zeta $ is a continuity point of the limit
distribution. This then completes the proof of part 2.

For part 3 we consider first the case $\zeta =\infty $. Clearly, $%
G_{A,n,\theta _{n}}(x)\;$converges to $\lim_{n}\Phi (w_{n,\theta
_{n}}^{(2)}(x))$. Since%
\begin{equation*}
w_{n,\theta _{n}}^{(2)}(x)=n^{1/2}\mu _{n}\left\{ (-\theta _{n}/\mu _{n}+x)+%
\sqrt{(\theta _{n}/\mu _{n}+x)^{2}+4}\right\} /2
\end{equation*}%
by (\ref{w}), and because $\theta _{n}/\mu _{n}\rightarrow \infty $, it is
easy to see that $w_{n,\theta _{n}}^{(2)}(x)$ converges to $\infty $ if $x>0$
and to $-\infty $ if $x<0$. The case where $\zeta =-\infty $ is proved
analogously.\hfill $\blacksquare $

\textbf{Proof of Theorem \ref{imp1}: }Let $\theta _{n}(\delta )$ be
short-hand for $-(t+\delta )/n^{1/2}$. Elementary calculations show that%
\begin{equation}
\lim_{\delta \downarrow 0}\left\vert F_{A,n,\theta _{n}(-\delta
)}(t)-F_{A,n,\theta _{n}(\delta )}(t)\right\vert =\Phi (t+n^{1/2}\mu
_{n})-\Phi (t-n^{1/2}\mu _{n}).  \label{z}
\end{equation}%
In particular, this implies that the supremum of $\left\vert F_{A,n,\theta
_{n}(-\delta )}(t)-F_{A,n,\theta _{n}(\delta )}(t)\right\vert $ over $0\leq
\delta <c-\left\vert t\right\vert $ is bounded from below by $\Phi
(t+n^{1/2}\mu _{n})-\Phi (t-n^{1/2}\mu _{n})$. The rest of the argument then
proceeds similar as in the proof of Theorem 13 in \cite{poelee07}.\hfill $%
\blacksquare $

\textbf{Proof of Theorem \ref{imp2}: }Analogous to the proof of Theorem 14
in \cite{poelee07} except for using (\ref{z}) in place of (11) in \cite%
{poelee07}.\hfill $\blacksquare $

\textbf{Proof of Theorem \ref{imp3}: }Analogous to the proof of Theorem 18
in \cite{poelee07}.\hfill $\blacksquare $

\bibliographystyle{biometrika}
\bibliography{stringsFULL,statbib}


\end{document}